\documentclass[11pt,a4paper]{article}

\usepackage{epsf,epsfig,amsfonts,amsgen,amsmath,amstext,amsbsy,amsopn,amsthm}
\usepackage{amsmath,cite}
\usepackage{amsfonts,amsthm,amssymb,bm}
\usepackage{amsfonts}
\usepackage{graphics}
\usepackage{latexsym,bm}
\usepackage{amsfonts,amsthm,amssymb,bbding}
\usepackage{indentfirst}
\usepackage{graphicx}
\usepackage{color}
\usepackage[colorlinks=true,anchorcolor=blue,filecolor=blue,linkcolor=red,urlcolor=blue,citecolor=blue]{hyperref} 
\usepackage{float}
\usepackage{tikz,enumerate}
\usepackage[margin=2.4cm]{geometry}

\newtheorem{thm}{Theorem}[section]
\newtheorem{prob}[thm]{Problem}

\newtheorem{lem}[thm]{Lemma}

\newtheorem{conj}[thm]{Conjecture}
\newtheorem{claim}[thm]{Claim}

\addtocounter{section}{0}

\begin{document}
\title{More on spectral supersaturation for the bowtie}

\author{
Longfei Fang\thanks{School of Mathematics and Finance, Chuzhou University, Chuzhou, Anhui 239012, China. Supported by the National Natural Science Foundation of China (No. 12501471). Email: \url{lffang@chzu.edu.cn}.} 
\and
Yongtao Li\thanks{Corresponding author. Yau Mathematical Sciences Center, Tsinghua University, Beijing 100084, China. 
Email: \url{ytli0921@hnu.edu.cn}.}
\and 
Huiqiu Lin\thanks{School of Mathematics, East China University of Science and Technology, Shanghai 200237, China. Supported by the National Natural Science Foundation of China (No. 12271162), and the Natural Science Foundation of Shanghai (No. 22ZR1416300). Email: \url{huiqiulin@126.com}.}
}

\date{\today}

\maketitle

\begin{abstract} 
A central topic in extremal graph theory is the supersaturation problem, which studies the minimum number of copies of a fixed substructure that must appear in any graph with more edges than the corresponding Turán number. 
Significant works due to Erd\H{o}s, Rademacher, Lov\'{a}sz and Simonovits investigated the supersaturation problem for the triangle. 
Moreover, Kang, Makai and Pikhurko studied the case for the bowtie, which consists of two triangles sharing a vertex. 
Building upon the pivotal results established by Bollob\'{a}s, Nikiforov, Ning and Zhai on counting triangles via the spectral radius, 
we study in this paper the spectral supersaturation problem for the bowtie. 
Let $\lambda (G)$ be the spectral radius of a graph $G$, and let $K_{\lceil \frac{n}{2}\rceil, \lfloor \frac{n}{2}\rfloor}^q$ be the graph obtained from Tur\'{a}n graph $T_{n,2}$ by adding $q$ pairwise disjoint edges to the partite set of size $\lceil \frac{n}{2}\rceil$. 
Firstly, we prove that 
there exists an absolute constant $\delta >0$ such that
if $n$ is sufficiently large, $2\le q \le \delta \sqrt{n}$, and $G$
 is an $n$-vertex graph with $\lambda (G)\ge \lambda (K_{\lceil \frac{n}{2}\rceil, \lfloor \frac{n}{2}\rfloor}^q)$, then $G$ contains at least ${q\choose 2}\lfloor \frac{n}{2}\rfloor$ bowties, and $K_{\lceil \frac{n}{2}\rceil, \lfloor \frac{n}{2}\rfloor}^q$ is the unique spectral extremal graph. This solves an open problem proposed by Li, Feng and Peng. 
 Secondly, we show that a graph $G$ whose spectral radius exceeds that of the spectral extremal graph for the bowtie must contain at least $\lfloor \frac{n-1}{2}\rfloor$ bowties. This sharp bound reveals a distinct phenomenon from the edge-supersaturation case, which guarantees at least $\lfloor \frac{n}{2}\rfloor$ bowties.
\end{abstract}

%%We will add the following keywords if the journal requests them. 

\iffalse 
\textbf{Keywords:} spectral radius; supersaturation; triangle; bowtie

\textbf{AMS Classification:} 05C35; 05C50
\fi

\section{Introduction}

The \emph{Tur\'{a}n number} ${\rm ex}(n,F)$ of a graph $F$ is defined as the maximum number of edges in an $n$-vertex graph that does not contain $F$ as a subgraph. 
The \emph{supersaturation problem} 
asks how many copies of $F$ are guaranteed in an $n$-vertex graph with $\mathrm{ex}(n,F)+q$ edges, where $q\ge 1$ is an integer.
As a foundation, Rademacher (see, e.g., \cite{Erdos1964})
showed that every $n$-vertex graph with $\lfloor\frac{n^2}{4}\rfloor+1$ edges
must contain at least $\lfloor \frac{n}{2} \rfloor$ triangles. This bound is optimal by adding one edge to the larger partite set of a  bipartite Tur\'{a}n graph $T_{n,2}$. 
Subsequently, Lov\'{a}sz and Simonovits \cite{Lovasz1975} extended this result 
by proving that if $1\leq q<\frac{n}{2}$ is an integer and $G$ is an $n$-vertex graph with $e(G)\geq\lfloor \frac{n^2}{4}\rfloor+q$,
then $G$ contains at least $q\lfloor \frac{n}{2} \rfloor$ triangles. 
In addition, the supersaturation problems for cliques and color-critical graphs have been extensively studied in the literature; 
see, e.g., cliques \cite{Rei2016, LPS2020, XK2021, LM2022-Erd-Rad, BC2023}, color-critical graphs \cite{Mubayi2010, Pikhurko2017} and references therein.

The supersaturation problem has been relatively  less explored for non-bipartite substructures that are not color-critical. To our knowledge, there are only works of \cite{KMP2020, MY2023} in the literature. Counting non-color-critical substructures appears more challenging and complicated. In this paper, we mainly address the supersaturation problem in this setting.

Let $F_k$ be the {\it friendship graph} obtained from $k$ triangles by sharing a common vertex. 
When $k=2$, we call $F_2$ the {\it bowtie}. 
 This simple configuration has been crucial in several areas and has been extensively studied in the literature; see \cite{Erdos1995,KMP2020,CFTZ2020,ZLX2022}. 
In 1995, 
Erd\H{o}s, F\"{u}redi, Gould and Gunderson \cite{Erdos1995} determined that 
$\mathrm{ex}(n,F_2)=\lfloor\frac{n^2}{4}\rfloor+1$ for every $n\geq 5$, and 
the extremal graphs are obtained from Tur\'{a}n graph $T_{n,2}$ by adding an edge. 
In 2020, Kang, Makai and Pikhurko \cite{KMP2020} studied the supersaturation problem for $F_2$ and  proved that there exists an absolute constant $\delta > 0$ such that for any $n\ge 1/ \delta$ and $1<q\le \delta n^2$, if $G$ is an $n$-vertex graph with 
 $\lfloor \frac{n^2}{4} \rfloor + q$ edges
and $G$ attains the minimum number of copies of $F_2$, then $G$ contains $T_{n,2}$ as a subgraph.

\subsection{Supersaturation via spectral radius}

 The {\it adjacency matrix} of a graph 
$G$ is defined as $A(G)=[a_{ij}]_{i,j=1}^n$,
where $a_{ij}=1$ if $ij\in E(G)$, and $a_{ij}=0$ otherwise.
The {\it spectral radius} $\lambda (G)$ is defined as the maximum modulus of eigenvalues of $A(G)$. 
Spectral graph theory lies at the intersection of graph theory and algebraic theory, e.g., employing the algebraic properties of matrices to investigate the structural properties of graphs. 
One popular topic is the extremal spectral graph theory, which bounds the eigenvalues of a family of graphs with certain properties. 
A bulk of results in extremal graph theory have been established to the spectral counterpart. In 1986, Wilf \cite{Wilf1986} proved that if $G$ is a $K_{r+1}$-free graph on $n$ vertices, then $\lambda (G)\le (1- \frac{1}{r})n$. We refer to \cite{BDT2025, Niki2002cpc, Nikiforov4, Niki2009cpc, Niki2009ejc,CDT2023,ZL2022jctb,LLZ-edge-spectral,LLZ-edge-color-critical,WKX2023} and references therein.

The supersaturation problem in terms of the spectral radius was initially studied by 
Bollob\'{a}s and Nikiforov \cite{Bollobas2007}, where they showed that
every $n$-vertex graph $G$ contains at least
$(\frac{\lambda (G)}{n}-1+\frac{1}{r}) \frac{r(r-1)}{r+1} (\frac{n}{r})^{r+1}$ 
copies of $K_{r+1}$. In particular, $G$ contains at least $\frac{n^2}{12}(\lambda -\frac{n}{2})$ triangles.  
Motivated by Rademacher's theorem, 
Ning and Zhai \cite{NZ2023} showed that if $G$ is an $n$-vertex graph with
$\lambda (G)\geq \lambda (T_{n,2})$,
then $G$ contains at least $\lfloor\frac{n}{2}\rfloor-1$  triangles, unless $G= T_{n,2}$. 
Moreover, Li, Feng and Peng \cite{LFP-spectral-LL} proved that for all integers $q\le \frac{1}{11}\sqrt{n}$, if $G$ is an $n$-vertex graph with $\lambda (G)\ge \lambda (K_{\lceil\frac{n}{2}\rceil, \lfloor\frac{n}{2}\rfloor}^{q})$, then $G$ contains at least $q\lfloor \frac{n}{2}\rfloor$ triangles. 
We refer to \cite{NZ2021-4-cycle,LFP2024-triangular,LFP-bowtie,LFP-spectral-LL,LLZ2024-book-4-cycle} for related results.

The spectral Tur\'{a}n problem for $F_k$ was recently investigated in \cite{CFTZ2020,ZLX2022} and \cite[Sec. 4.3]{LFP2024-triangular}. 
Following these works, we mainly study the spectral supersaturation problem for the bowtie $F_2$. 
Let $K_{s,t}^{q}$ be the graph obtained from $K_{s,t}$ by adding $q$ {\it pairwise disjoint} edges into the vertex part of size $s$. Recently,   
Li, Lu and Peng \cite{Li2023} determined the spectral extremal graph for $F_2$. 

\begin{thm}[Li--Lu--Peng \cite{Li2023}]
\label{thm-LLP}
For every integer $n\geq 7$, 
if $G$ is an $F_2$-free graph on $n$ vertices,
then 
 $\lambda (G)\leq \lambda (K_{\lfloor\frac{n}{2}\rfloor,\lceil\frac{n}{2}\rceil}^{1})$, 
with equality if and only if $G=K_{\lfloor\frac{n}{2}\rfloor,\lceil\frac{n}{2}\rceil}^{1}$.
\end{thm}

Throughout the paper, 
let $\tau (G)$ be the number of copies of $F_2$ in $G$. Li, Feng and Peng \cite{LFP-bowtie} proposed the following supersaturation problem for $F_2$ in terms of the spectral radius. 

\begin{conj}[Li--Feng--Peng \cite{LFP-bowtie}] 
\label{conj1.1A}
If $q \geq 2$ and $G$ is a graph with large order $n$ and
$$ \lambda (G) \geq \lambda (K_{\lceil\frac{n}{2}\rceil, \lfloor\frac{n}{2}\rfloor}^{q}), $$ 
then $\tau(G)\geq\binom{q}{2}\lfloor\frac{n}{2}\rfloor$,
with equality if and only if $G= K_{\lceil\frac{n}{2}\rceil,\lfloor\frac{n}{2}\rfloor}^{q}$.
\end{conj}

Conjecture \ref{conj1.1A} implies that if  $\lambda (G) \geq \lambda (K_{\lceil\frac{n}{2}\rceil, \lfloor\frac{n}{2}\rfloor}^{q})$ and $G$ minimizes $\tau (G)$, then $G$ contains $T_{n,2}$ as a subgraph. As a starting point, 
Li, Feng and Peng \cite{LFP-bowtie} confirmed the case $q=2$.

\subsection{Main results}

In this paper,
we settle Conjecture \ref{conj1.1A} in the following stronger sense. 

\begin{thm}\label{THM1.6A} 
There exists an absolute constant $\delta >0$ such that for any sufficiently large $n$ and $2\le q \le \delta \sqrt{n}$, the following holds: If $G$ is an $n$-vertex graph with $\lambda (G)\geq \lambda (K_{\lceil\frac{n}{2}\rceil,\lfloor\frac{n}{2}\rfloor}^{q})$,
then $\tau(G)\geq \binom{q}{2}\lfloor\frac{n}{2}\rfloor$,
where the equality holds if and only if $G= K_{\lceil\frac{n}{2}\rceil,\lfloor\frac{n}{2}\rfloor}^{q}$.
\end{thm}

The edge-supersaturation result of Kang, Makai and Pikhurko \cite{KMP2020} implies that 
if $n$ is sufficiently large and $e(G)>\lfloor \frac{n^2}{4}\rfloor +1$, then $\tau (G)\ge \lfloor \frac{n}{2} \rfloor$. 
It is natural to consider the correspondence under a spectral  assumption. For example, 
when $\lambda (G)> \lambda (K_{\lfloor\frac{n}{2}\rfloor,\lceil\frac{n}{2}\rceil}^{1})$, 
is it true that $\tau (G) \ge \lfloor\frac{n}{2}\rfloor$?
Unfortunately, the answer turns out to be negative.  
In the sequel, we show that $\tau (G)\ge \lfloor \frac{n-1}{2}\rfloor$, and this bound is optimal. 
This reveals a different phenomenon from the edge-supersaturation. 

Let $K_{s,t}^{2, \Gamma}$ (resp. $K_{s,t}^{2,1}$) denote the graph obtained from $K_{s,t}^{2}$ by deleting a cross-edge that is incident to an edge (resp. not incident to any edge) of the two  class-edges; see Figure \ref{fig-A2D2}. 

 \begin{figure}[htbp]
\centering
\includegraphics[scale=0.85]{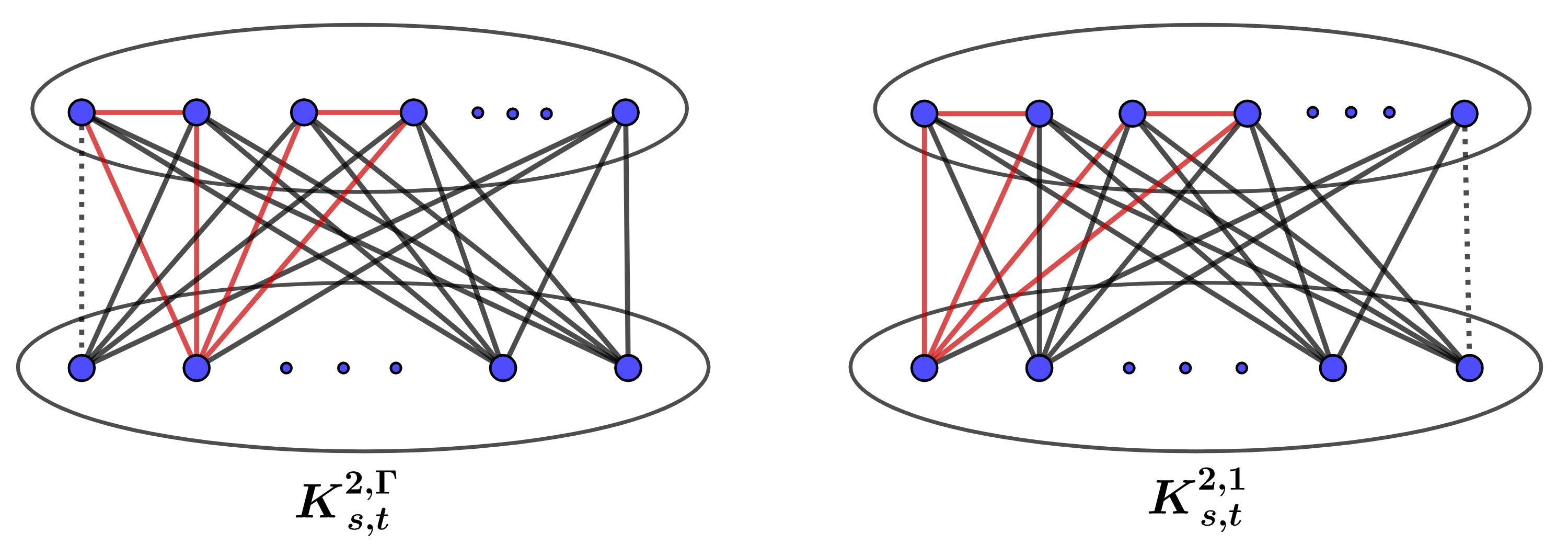} 
\caption{The graphs $K_{\,s,t}^{2,\Gamma}$ and 
$K_{\,s,t}^{2,1}$.}
\label{fig-A2D2}
\end{figure}

 \begin{thm}\label{THM1.4A} 
Let $n$ be sufficiently large. For every $n$-vertex graph $G$ with $\lambda (G)\ge \lambda (K_{\lfloor\frac{n}{2}\rfloor,\lceil\frac{n}{2}\rceil}^{1})$,
we have $\tau(G)\geq\lfloor\frac{n-1}{2}\rfloor$, unless $G= K_{\lfloor\frac{n}{2}\rfloor,\lceil\frac{n}{2}\rceil}^{1}$.
Moreover, the bound is achieved if and only if
\begin{enumerate}
\item[\rm (i)] for even $n$,
$G$ is either $K_{\frac{n}{2}+1,\frac{n}{2}-1}^{2}$ or $K_{\frac{n}{2},\frac{n}{2}}^{2, \Gamma}$;

\item[\rm (ii)] for odd $n$, $G$ is any one of $K_{\frac{n+1}{2},\frac{n-1}{2}}^{2},K_{\frac{n+1}{2},\frac{n-1}{2}}^{2,1} $ or 
$ K_{\frac{n-1}{2},\frac{n+1}{2}}^{2, \Gamma}$.
\end{enumerate}
\end{thm}

Theorem \ref{THM1.4A} extends Theorem \ref{thm-LLP} to the spectral supersaturation setting.

Given a vertex partition $V(G)=V_1\cup V_2$, we call an edge $e\in E(G)$ a class-edge if $e\in E(G[V_i])$ for some $i\in \{1,2\}$, and a cross-edge otherwise.  The key ingredient in our proofs of both Theorem \ref{THM1.6A} and Theorem \ref{THM1.4A} is the following  structural characterization for graphs that have large spectral radius and contain a small number of copies of the bowtie.  

\begin{thm}\label{second-key} 
There exists an absolute constant $\delta >0$ such that for sufficiently large $n$ and $2\le q \le \delta \sqrt{n}$, the following holds.
If $G$ is an $n$-vertex graph with 
$\lambda (G)\ge \lambda (K_{\lfloor\frac{n}{2}\rfloor,\lceil\frac{n}{2}\rceil}^{1})$ and 
    $\tau (G)\le {q \choose 2} \lceil \frac{n}{2} \rceil$, 
then $G$ can be obtained from $K_{n_1,n_2}$ by adding $\alpha_1$ class-edges and deleting $\alpha_2$ cross-edges, where $\alpha_2< \alpha_1 \le q$. Under the above conditions, if $n$ is even, then $(n_1-n_2)^2\le 8 (\alpha_1-\alpha_2)-8$; if $n$ is odd, then $(n_1-n_2)^2\le 8 (\alpha_1-\alpha_2)-7$.   
\end{thm}

We point out that the advantage of Theorem \ref{second-key} lies in providing $\alpha_2 < \alpha_1\le q$ and bounding the gap between $n_1$ and $n_2$ accurately. These bounds significantly improve those established in the graph stability process described in \cite{LFP-bowtie}. In particular, they play a crucial role in proving Theorem \ref{THM1.4A}, as they allow for the efficient exclusion of certain exceptional graphs.

\medskip
The remainder of this paper is organized as follows.  
In Section \ref{section2A}, we present some preliminary results. In Section \ref{section3}, we give the proof of the  key ingredient Theorem \ref{second-key}. In Section \ref{section4}, we prove Theorems \ref{THM1.6A} and \ref{THM1.4A}. 
In Section \ref{sec-Concluding}, we conclude some problems for interested readers.

\section{Preliminaries}\label{section2A}

We introduce some preliminary results for our proofs.
Given a simple graph $G$,
we use $|G|$ to denote the number of vertices of $G$. 
 We need to use the following supersaturation result. 

\begin{lem}[See \cite{ES1983}] 
\label{lem-ES-super}
    Let $F$ be a graph with $\chi (F)=r+1$. Then for any $\varepsilon >0$, there exist $n_0=n_0(F,\varepsilon)$ and $\delta = \delta (F,\varepsilon)>0$ such that if $G$ is a graph on $n\ge n_0$ vertices with $e(G)\ge (1-\frac{1}{r} + \varepsilon )\frac{n^2}{2}$, then $G$ contains at least $\delta n^{|F|}$ copies of $F$, where $|F|$ denotes the order of $F$.   
\end{lem}
 
The following supersaturation-stability result can be found in \cite[Theorem 2.8]{FLLM2025}. 
 
\begin{thm}[See \cite{FLLM2025}] \label{thm-sss}
Let $F$ be a graph with $\chi (F)=r+1$. 
For any $\varepsilon >0$, 
there exist $\eta >0, \delta >0$ and $n_0$ such that if $G$ 
is a graph on $n\ge n_0$ vertices with at most 
$\eta n^{|F|}$ copies of $F$ and 
$ \lambda (G) \ge (1-\frac{1}{r} -\delta )n$, 
then $G$ can be obtained from $T_{n,r}$ by adding and deleting $\varepsilon n^2$ edges.
\end{thm} 

We denote by $K_r(n_1,\dots,n_r)$ the complete $r$-partite graph with color classes $V_1,\ldots ,V_r$, where $|V_i|=n_i$ for every $i\in [r]$. 
The following result \cite[Theorem 1.5]{FLLM2025} bounds the increment of the spectral radius of a graph that differs from a complete multipartite graph within few edges. 

\begin{thm}[See \cite{FLLM2025}] 
\label{first-key}
Let $n$ be sufficiently large 
and $G$ be obtained from an $n$-vertex complete $r$-partite graph $K=K_r(n_1,\dots,n_r)$ by
 adding $\alpha_1$ class-edges and deleting $\alpha_2$ cross-edges,
where $n_1\geq n_2\geq \cdots \geq n_r$ and $\max\{\alpha_1,\alpha_2\}\le \frac{n}{(20r)^3}$. 
Then the following statements hold:
\begin{enumerate}  
\item[\rm (i)] If $n_1-n_r\le \frac{n}{400}$, then by denoting $\phi =\max\{n_1-n_r, 2(\alpha_1+\alpha_2)\}$, we have 
\[ \left| \lambda (G)- \lambda (K)- \frac{2(\alpha_1-\alpha_2)}{n} \right| \le \frac{56(\alpha_1+\alpha_2)\phi}{n^{2}}. \]

\item[\rm (ii)] If $n_1-n_r\geq 2k$ for an integer $k\le \frac{n}{(20r)^3}$, then by denoting $\psi=\max\{3k,2(\alpha_1+\alpha_2)\}$, 
\[ \lambda (G)\leq
    \lambda (T_{n,r})+\frac{2(\alpha_1-\alpha_2)}{n} 
    -\frac{2(r-1)k^2}{rn}\Big( 1 - \frac{28r\psi}{n}\Big)^4 + 
    \frac{56(\alpha_1+\alpha_2)\cdot 7r\psi}{n^2}.\] 
\end{enumerate}
\end{thm}

Using Theorem \ref{first-key},
we get the following lemma immediately. 

\begin{lem} \label{low-bound+1}
For sufficiently large $n$, we have 
    \begin{equation*}
 \lambda (K_{\lfloor\frac{n}{2}\rfloor,\lceil\frac{n}{2}\rceil}^{1})
=\begin{cases}
\frac{n}{2}+\frac{2}{n}+O(\frac{1}{n^2}) & \hbox{if $n$ is even;} \\
  \frac{n}{2}+\frac{7}{4n}+O(\frac{1}{n^2}) & \hbox{if $n$ is odd}.
\end{cases}
\end{equation*} 
\end{lem}
  
Next, we present detailed comparisons of spectral radii of some specific graphs. 

\begin{lem}\label{lem-even-case}
For sufficiently large even $n$, we have $\lambda (K_{\frac{n}{2}+1,\frac{n}{2}-1}^{2}),\lambda (K_{\frac{n}{2},\frac{n}{2}}^{2,\Gamma}) > 
\lambda (K_{ \frac{n}{2},\frac{n}{2}}^1)$.
\end{lem}

\begin{proof}
For notational convenience, we denote $\lambda = \lambda (K_{ \frac{n}{2},\frac{n}{2}}^1)$. 
Let $\bm{x}=(x_1,\ldots,x_n)^{\mathrm{T}}$ be the positive unit eigenvector of $K_{\frac{n}{2},\frac{n}{2}}^{1}$.
We partition the vertex set of $K_{\frac{n}{2},\frac{n}{2}}^{1}$ as $\Pi$:%
{\small \begin{align*}
V(K_{\frac{n}{2},\frac{n}{2}}^{1})=U_{1}\cup U_{2}\cup U_3,
\end{align*}}%
where $U_{1}$ induces a class-edge of $K_{\frac{n}{2},\frac{n}{2}}^{1}$, and 
$U_1\cup U_2$ and $U_3$ are partite sets of $K_{\frac{n}{2},\frac{n}{2}}$
satisfying $|U_1|+|U_2|=|U_3|=\frac{n}{2}$.
For every $i\in [3]$, by symmetry, we know that $x_u=x_v$ for any distinct vertices $u,v\in U_i$.
Therefore, we may assume that $x_v=x_i$ for each $v\in U_i$.
Then
$$\begin{cases}
   \lambda  x_1=x_1+\frac{n}{2}x_3, \\
 \lambda  x_2=\frac{n}{2}x_3,  \\
 \lambda  x_3=2 x_1+\frac{n-4}{2}x_2.
\end{cases}
$$ 
Thus, we see that $\lambda$ is the largest eigenvalue of
\[
B_{\Pi} = \begin{bmatrix}
1 & 0 & \frac{n}{2} \\
0 & 0 & \frac{n}{2} \\
2 & \frac{n-4}{2} & 0
\end{bmatrix}.
\]
Then $\lambda$
is the largest root of
\begin{align*}
f_1(x) :=\det(xI_3-B_{\Pi} )= 
x^3-x^2-{(n^2x)}/{4} + n^2/4 -n.
\end{align*}  
By Lemma \ref{low-bound+1},  we know that
$\frac{n}{2}+\frac{3}{2n}\leq \lambda \leq \frac{n}{2}+\frac{5}{2n}$ for sufficiently large $n$. 

By a similar computation as above, we obtain that $\lambda (K_{\frac{n}{2}+1,\frac{n}{2}-1}^{2})$
is the largest root of
\begin{align*}
f_2(x) := x^3-x^2
-({n^2/4-1})x+{n^2/4-2n+3}.
\end{align*}   
We can check that 
$f_2(x)-f_1(x)=-n+3+x<0$ for every $x\in [\frac{n}{2}+\frac{3}{2n},\frac{n}{2}+\frac{5}{2n}]$,
which implies that $f_2(\lambda )< f_1(\lambda) = 0$. 
Thus, we get $\lambda (K_{\frac{n}{2}+1,\frac{n}{2}-1}^{2})>\lambda $. 

Similarly, we know that $\lambda (K_{\frac{n}{2},\frac{n}{2}}^{2, \Gamma})$
is the largest root of
\begin{align*}
f_3(x) := x^6 - x^5 - {(n^2x^4)}/{4} 
+ ({n^2/4-2n+2})x^3 + ({n^2/2-n-1})x^2 
-({n^2/2-4n+3})x.  
\end{align*}  
Then we have  
$f_3(x)-x^3f_1(x)=-(n-2)x^3+({n^2/2-n-1})x^2
-({n^2/2-4n+3})x$.  
For every $x\in [\frac{n}{2}+\frac{3}{2n},\frac{n}{2}+\frac{5}{2n}]$,
we can verify that 
$ f_3(x)-x^3f_1(x)=-\frac{1}{4}n^3+O(n^2)<0,$
which implies that $f_3(\lambda)< \lambda^3f_1(\lambda)= 0$.
Thus, we have $\lambda (K_{\frac{n}{2},\frac{n}{2}}^{2, \Gamma})>\lambda $, as desired. 
\end{proof}

\begin{lem}\label{lem-odd-case}
For sufficiently large odd integer $n$, we have 
\begin{itemize}
\item[\rm (i)] $\lambda (K_{\frac{n+1}{2},\frac{n-1}{2}}^{2}),
\lambda (K_{\frac{n+1}{2},\frac{n-1}{2}}^{2,1}),
\lambda (K_{\frac{n-1}{2},\frac{n+1}{2}}^{2, \Gamma})>\lambda (K_{\frac{n-1}{2},\frac{n+1}{2}}^1)$;

\item[\rm (ii)] $\lambda (K_{\frac{n+3}{2},\frac{n-3}{2}}^{2}),
\lambda (K_{\frac{n+1}{2},\frac{n-1}{2}}^{2, \Gamma})<\lambda (K_{\frac{n-1}{2},\frac{n+1}{2}}^1)$.
\end{itemize}
\end{lem}

\begin{proof} 
We denote $\lambda := \lambda (K_{\frac{n-1}{2},\frac{n+1}{2}}^1)$.
It is similar to get that
$\lambda$ is the largest root of
\begin{align*}
g_1(x) := x^3-x^2 
-({n^2/4-1/4})x+{n^2/4-n-5/4}.
\end{align*}  
By Lemma \ref{low-bound+1}, we know that
$\frac{n}{2}+\frac{3}{2n}\leq \lambda \leq \frac{n}{2}+\frac{2}{n}$ for sufficiently large $n$.

(i) By computation, we obtain that $\lambda (K_{\frac{n+1}{2},\frac{n-1}{2}}^{2})$
is the largest root of
\begin{align*}
g_2(x) :=x^3-x^2 
-({n^2/4-1/4})x+{n^2/4-2n+7/4}.
\end{align*}  
Then
$g_2(x)-g_1(x)=-n+3<0$ for $x\in [\frac{n}{2}+\frac{3}{2n},\frac{n}{2}+\frac{2}{n}]$.
So $\lambda (K_{\frac{n+1}{2},\frac{n-1}{2}}^{2})>\lambda $. 
In addition, this inequality can also be seen by applying part (i) of Theorem \ref{first-key}. 

%Indeed, Theorem \ref{first-key} (i) gives $\lambda (K_{\frac{n+1}{2},\frac{n-1}{2}}^{2})= \lambda (K_{\frac{n+1}{2},\frac{n-1}{2}}) + \frac{4}{n} + O(\frac{1}{n^2})$, while $\lambda (K_{\frac{n+1}{2},\frac{n-1}{2}}^1) = \lambda (K_{\frac{n+1}{2},\frac{n-1}{2}}) + \frac{2}{n} + O(\frac{1}{n^2})$. 

By computation, we obtain that $\lambda (K_{\frac{n+1}{2},\frac{n-1}{2}}^{2,1})$
is the largest root of
\begin{align*}
g_3(x) &:=x^5-x^4-({n^2-5})x^3/4 
+ ({n^2- 8n+ 3})x^2/4 \\
     &\quad + ({n^2- 4n + 3})x/4 -({n^2-12n+ 27})/{4}.
\end{align*}  
Then 
\begin{align*}
g_3(x)-x^2g_1(x)=x^3-(n-2)x^2 
+({n^2-4n+3})x/4- ({n^2-12n+27})/{4}.
\end{align*}  
For $x\in [\frac{n}{2}+\frac{3}{2n},\frac{n}{2}+\frac{2}{n}]$,
we have $g_3(x)-x^2g_1(x)=-\frac14 n^2+o(n^2)<0$, which implies that $g_3(\lambda)< \lambda^2g_1(\lambda)= 0$.
Thus, we conclude that $\lambda (K_{\frac{n+1}{2},\frac{n-1}{2}}^{2,1})>\lambda $.

By computation, 
we obtain that $\lambda (K_{\frac{n-1}{2},\frac{n+1}{2}}^{2, \Gamma})$
is the largest root of
\begin{align*}
g_4(x) &:=x^6 - x^5 -({n^2-1})x^4/4 
+ ({n^2- 8n-1})x^3/4 \\
      &\quad +({n^2- 2n- 3})x^2/2 
      -({n^2-8n-1})x/2.
\end{align*}  
Then 
\begin{align*}
g_4(x)-x^3g_1(x)=-(n-1)x^3+({n^2-2n-3})x^2/2 
-({n^2-8n-1})x/{2}.
\end{align*}  
We see that $g_4(x)-x^3g_1(x)=-\frac{3}{8}n^3+o(n^3)<0$ for $x\in [\frac{n}{2}+\frac{3}{2n},\frac{n}{2}+\frac{2}{n}]$.
So $\lambda (K_{\frac{n-1}{2},\frac{n+1}{2}}^{2, \Gamma})>\lambda $.

(ii)
We denote $\lambda^* :=\lambda (K_{\frac{n+3}{2},\frac{n-3}{2}}^{2})$ and $\lambda^{**}:=\lambda (K_{\frac{n+1}{2},\frac{n-1}{2}}^{2, \Gamma})$.
By Theorem \ref{first-key} (i),
we have
\begin{align*}
\lambda^*
=\lambda (K_{\frac{n+3}{2},\frac{n-3}{2}})+\frac{4}{n}+O\Big(\frac{1}{n^{2}} \Big)
=\frac{n}{2}+\frac{7}{4n}+O\Big(\frac{1}{n^{2}}\Big),
\end{align*}  
which implies that $\frac{n}{2}+\frac{3}{2n}\leq \lambda^*\leq \frac{n}{2}+\frac{2}{n}$ for large $n$.
By computation, $\lambda^*$
is the largest root of
\begin{align*}
g_5(x) := x^3-x^2 -({n^2-9})x/4 
+({n^2-8n+15})/{4}.
\end{align*}  
For $x\in [\frac{n}{2}+\frac{3}{2n},\frac{n}{2}+\frac{2}{n}]$,
we have $g_1(x)-g_5(x)=n-5-2x<0$.
Then $g_1(\lambda^*)<0$ and $\lambda^* < \lambda$.

By Theorem \ref{first-key} (i) again,
we see that 
\begin{align*}
\lambda^{**}
=\lambda (K_{\frac{n+1}{2},\frac{n-1}{2}})+\frac{2}{n}+O\Big(\frac{1}{n^{2}}\Big)
=\frac{n}{2}+\frac{7}{4n}+O\Big(\frac{1}{n^{2}}\Big),
\end{align*}  
which implies that $\frac{n}{2}+\frac{3}{2n}\leq \lambda^{**}\leq \frac{n}{2}+\frac{2}{n}$ for sufficiently large $n$.
By computation, 
we obtain that $\lambda^{**}$
is the largest root of
\begin{align*}
g_6(x) &:=x^6 - x^5 -({n^2-1})x^4/4 + 
({n^2 - 8n + 15})x^3/4 \\
   & \quad + ({n^2- 2n- 3})x^2/2 
   -({n^2 -8n+11})x/2.
\end{align*}  
Then we compute that  
\begin{align*}
x^3g_1(x)-g_6(x) =(n-5)x^3- 
({n^2-2n-3})x^2/2+ ({n^2-8n+11})x/2.
\end{align*}  
It is easy to verify that $x^3g_1(x)-g_6(x) =-\frac{n^3}{8}+o(n^3)<0$ for $x\in [\frac{n}{2}+\frac{3}{2n},\frac{n}{2}+\frac{2}{n}]$.
It follows that $g_1(\lambda^{**})<0$, which implies  $\lambda^{**} < \lambda $, as desired. 
\end{proof}

\section{Proof of Theorem \ref{second-key}}
\label{section3}

The following lemma provides a tool for counting the copies of $F_2$ in a graph close to a complete bipartite graph. 
This serves as an important tool in the proof of Theorems \ref{THM1.6A} and \ref{THM1.4A}.

\begin{lem} \label{LEM2.9A}
Let $n$ be sufficiently large, and $G$ be a graph obtained from an $n$-vertex
complete bipartite graph $K_{n_1,n_2}$ by adding $\alpha_1\geq 2$ class-edges
and deleting $\alpha_2$ cross-edges of $K_{n_1,n_2}$.
If $\alpha_2<\alpha_1< \frac{1}{3}\sqrt{n}$ and
$k:=|n_1-n_2|< 4\sqrt{\alpha_1}$,
then 
$\tau(G)\geq \binom{\alpha_1}{2}\frac{n-k}{2}-  \alpha_1\alpha_2$, 
with equality if and only if all $\alpha_1$ edges are pairwise disjoint and lie in exactly one partite set of $K_{n_1,n_2}$.
\end{lem}

\begin{proof}
Let $V_1,V_2$ be the partite sets of $K=K_{n_1,n_2}$ with size $n_1,n_2$, respectively. 
We may assume that $G$ achieves the minimum number of copies of $F_2$. 
For brevity, we write $e(V_i)$ instead of $e(G[V_i])$ for each $i\in \{1,2\}$.
By symmetry, we may assume that $e(V_{1})\geq e(V_{2})$. Our goal is to show that $e(V_2)=0$. 
Let $G'$ be a graph obtained from $K_{n_1,n_2}$ by  adding $\alpha_1$ pairwise disjoint class-edges
into exactly one partite set of $K_{n_1,n_2}$,
and then removing $\alpha_2$ cross-edges from $K_{n_1,n_2}$. Without loss of generality, we may assume that $n_1 \ge n_2$. Since $k:=n_1-n_2$, we have $n_1=\frac{n+k}{2}$ and $n_2= \frac{n-k}{2}$. Observe that removing a cross-edge destroys at most $\alpha_1$ copies of $F_2$. 
Then 
\begin{align}\label{eq22}
\tau(G)\leq \tau(G')=\binom{\alpha_1}{2}\frac{n-k}{2} - \alpha_1\alpha_2 . 
\end{align}

\begin{claim}\label{Cla2.1B}
For every $i\in \{1,2\}$, the edges in $G[V_i]$ are pairwise disjoint.
\end{claim}

\begin{proof}[Proof of claim]
By way of contradiction, we may assume that $G[V_1]$ contains a copy of $P_3$,
say $v_1uv_2$. 
In the complete bipartite graph $K_{n_1,n_2}$, 
any distinct vertices $w_1,w_2 \in V_2$, together with $v_1uv_2$, produce a copy of $F_2$. Since $n_2 = \frac{n-k}{2}$ and $k=O(\sqrt{n})$, 
after removing $\alpha_2$ cross-edges from $K_{n_1,n_2}$, 
we see that $v_1,u,v_2$ have at least $n_2- \alpha_2$ common neighbors in $V_2$, and then  $\tau (G)\ge {n_2-\alpha_2 \choose 2} \ge {n/3 \choose 2}> \frac{n^2}{20}$. 
Since $\alpha_1 \le \frac{1}{3} \sqrt{n}$, we get $\tau (G)>  \frac{1}{4}\alpha_1^2 n$.  
This leads to a contradiction with \eqref{eq22}. 
We conclude that $G[V_i]$ is $P_3$-free for each $i\in \{1,2\}$. 
\end{proof}

\iffalse 
{\textbf{Case 1.}} $\alpha_2=0$.

By Claim \ref{Cla2.1B}, we know that $G[V_i]$ consists of pairwise disjoint edges and isolated vertices for each $i\in \{1,2\}$.
Recall that $e(V_1)+e(V_2)=\alpha_1$. 
If $e(V_2)\geq 1$, then we see that
 \begin{align*}
\tau(G)
&\ge \left(\binom{e(V_1)}{2}+\binom{e(V_2)}{2}+4e(V_1)e(V_2)\right) \frac{n -k}{2} \\
&= \left(\binom{e(V_1)+e(V_2)}{2}+3e(V_1)e(V_2)\right) \frac{n-k}{2}\\
&\geq \left(\binom{\alpha_1}{2}+3(\alpha_1-1)\right) 
\frac{n}{2}- \frac{1}{2}\alpha_1^2k,
\end{align*}  
which contradicts with \eqref{eq22}. 
Thus, we must have $e(V_{2})=0$, that is, all $\alpha_1$ edges are added within the vertex set $V_1$. Then 
$\tau(G)\ge \binom{\alpha_1}{2} \frac{n-k}{2}  
\ge \binom{\alpha_1}{2}\frac{n}{2}- \frac{1}{4}\alpha_1^2k$,
as required. 
\fi

In what follows, we prove that $e(V_2)=0$. 
Suppose on the contrary that $e(V_{2})\geq 1$.
Let $G''$ be the graph obtained from $K_{n_1,n_2}$ by adding $\alpha_1$ pairwise disjoint class-edges to $V_1$ and $V_2$. 
Observe that any pair of disjoint edges of $V_1$ is contained in $|V_2|$ bowties, and similarly any pair of disjoint edges of $V_2$ is contained in $|V_1|$ bowties. Moreover, for any two edges in different color classes, we can find $2(n-4)$ bowties in $G''$. Totally, it follows that  
 \begin{align}
\tau(G'')
&\ge \left(\binom{e(V_1)}{2}+\binom{e(V_2)}{2}\right) \frac{n -k}{2} + e(V_1)e(V_2)\cdot 2(n-4) \notag \\
&= \binom{e(V_1)+e(V_2)}{2}\frac{n-k}{2} + e(V_1)e(V_2)\left( 2(n-4) -\frac{n-k}{2} \right) \notag \\
&\ge \binom{\alpha_1}{2} 
\frac{n-k}{2} + (\alpha_1-1) 
\left(\frac{3n}{2} -8 \right), \label{eq22BBCD}
\end{align} 
By Claim \ref{Cla2.1B}, we see that 
the edges of $G$ within $V_1$ (and $V_2$) are pairwise disjoint. 
Every cross-edge of $G''$ can be incident with at most one class-edge of $V_1$ 
and at most one class-edge of $V_2$.  
Removing a cross-edge from $G''$ destroys at most $ (|V_1|-2) + (|V_2| -2) = n-4$ copies of $F_2$.  
So we have $\tau(G)\geq  \tau(G'') - \alpha_2 (n-4)$, which together with \eqref{eq22BBCD} gives 
 \begin{align}\label{eq22BBC}
\tau(G)\ge \binom{\alpha_1}{2} 
\frac{n-k}{2} + (\alpha_1-1) 
\left(\frac{3n}{2} -8 \right) - \alpha_2(n-4).
\end{align}  
Since $\alpha_2 \le \alpha_1-1$ and $n$ is large enough,  
we see from (\ref{eq22BBC}) that $\tau(G) > \binom{\alpha_1}{2} \frac{n-k}{2}$,
which contradicts with \eqref{eq22}.
Thus, we must have $e(V_{2})=0 $. 
We conclude that all $\alpha_1$ edges are added within the vertex set $V_1$. Consequently, we get  
$ \tau(G)\ge \binom{\alpha_1}{2} \frac{n-k}{2} -  \alpha_1\alpha_2$, as needed.
\end{proof}

Now, we are ready to prove the structural result of Theorem \ref{second-key}.

\begin{proof}[{\bf Proof of Theorem \ref{second-key}}] 
Let $\delta = {10^{-4}}$. Assume that $n$ is sufficiently large and $2\le q \le \delta \sqrt{n}$. 
Suppose that $G$ is a graph on $n$ vertices with 
$\lambda (G)\ge \lambda (K_{\lfloor\frac{n}{2} \rfloor, \lceil\frac{n}{2} \rceil}^1)$ and $\tau (G)\le {q\choose 2} \lfloor \frac{n}{2}\rfloor$. Then Lemma \ref{low-bound+1} yields the following: 
    \begin{equation}\label{EQU009DD}
\lambda (G)\geq \begin{cases}
\frac{n}{2}+\frac{2}{n}+O(\frac{1}{n^2}) & \hbox{if $n$ is even;} \\
  \frac{n}{2}+\frac{7}{4n}+O(\frac{1}{n^2}) & \hbox{if $n$ is odd}.
\end{cases}
\end{equation} 

Let $\varepsilon<10^{-9}$ be a fixed positive real  number. 
Note that $\lambda (G)\geq \frac{n}{2}$ and 
$\tau(G)\leq \binom{q}{2}\lceil\frac{n}{2}\rceil=o(n^5)$ since $q\le \delta \sqrt{n}$.  Applying Theorem \ref{thm-sss}, we know that for sufficiently large $n$, the graph 
$G$ can be obtained from Tur\'{a}n graph $T_{n,2}$ by adding and deleting at most $\varepsilon n^2$ edges.

\begin{claim}\label{CLA3.2}
Let $V_1\cup V_2$ be a partition of $V(G)$ such that $e(V_1,V_2)$ is maximized.
Then $e(V_1)+e(V_2)\leq \varepsilon n^2$ and $\big||V_i|-\frac{n}{2}\big|\leq\varepsilon^{\frac13} n$
for each $i\in \{1,2\}$.
\end{claim}

\begin{proof}[Proof of claim]
Since $G$ differs from $T_{n,2}$ in at most $\varepsilon n^2$ edges, we have
\begin{equation}\label{EQU009}
e(G)\geq e(T_{n,2})- \varepsilon n^2= \Big\lfloor\frac{n^2}{4}\Big\rfloor-\varepsilon n^2
>\frac{n^2}{4}-2\varepsilon n^2,
\end{equation}  
and there exists a partition $V(G)=U_1\cup U_2$ such that
$e(U_1)+e(U_2)\leq \varepsilon n^2$
and $\big\lfloor\frac n2\big\rfloor= |U_1|\leq |U_2|
=\big\lceil\frac n2\big\rceil$.
Now we select a new partition $V(G)=V_1\cup V_2$
such that $e(V_1,V_2)$ is maximized.
Equivalently, $e(V_1)+e(V_2)$
is minimized. Hence
$e(V_1)+e(V_2)\leq e(U_1)+e(U_2)\leq \varepsilon n^2.$
We denote $\sigma=|V_1|-\frac n2$.
Consequently, it follows that 
\begin{align*}
e(G)= |V_1||V_2|+e(V_1)+e(V_2)\leq \frac{n^2}{4}-\sigma^2+ \varepsilon n^2.
\end{align*}   
Combining \eqref{EQU009} yields $\sigma^2<3\varepsilon n^2$.
Therefore, we get $|\sigma|<\varepsilon^{\frac13}n$
as $\varepsilon<10^{-5}$.
\end{proof}

\begin{claim}\label{CLA3.3}
Let $S:=\{v\in V(G): d_G(v)\leq
\big(\frac{1}{2}-8\varepsilon^{\frac13}\big)n\}.$
Then $|S|\leq \varepsilon^{\frac13} n$.
\end{claim}

\begin{proof}[Proof of claim]
Suppose on the contrary that $|S|>\varepsilon^{\frac13} n$,
then there exists a subset $S_0\subseteq S$
with $|S_0|=\lfloor\varepsilon^{\frac13} n\rfloor$.
Set $n_0:=|V(G)\setminus S_0|=n-\lfloor\varepsilon^{\frac13} n\rfloor$. 
Thus, we have 
\begin{align*}
\frac{1}{4} n_0^2<\frac{1}{4}\big((1-\varepsilon^{\frac13})n+1\big)^2\leq \big(\frac{1}{4}-\frac{1}{2}\varepsilon^{\frac13}
+\varepsilon^{\frac23}\big)n^2
\end{align*}   
for $n$ sufficiently large.
Combining this with \eqref{EQU009}, we obtain
 \begin{align*}
e(G-S_0)&\geq  e(G)-\sum_{v\in S_0}d_G(v)
\geq \big\lfloor\frac{n^2}{4}\big\rfloor-2\varepsilon n^2-\varepsilon^{\frac13} n\Big(\frac{1}{2}-8\varepsilon^{\frac13}\Big)n\nonumber\\
&>\Big(\frac{1}{4}-\frac{1}{2}\varepsilon^{\frac13}
+8\varepsilon^{\frac23}-3\varepsilon\Big)n^2
>\frac{1}{4}n_0^2+2\varepsilon n_0^2,
\end{align*}  
where the last inequality follows by  $\varepsilon^{\frac23}
>\varepsilon$ and $n>n_0$.
Applying Lemma \ref{lem-ES-super}, we find that $G-S_0$ contains at least $\delta_{\ref{lem-ES-super}} n_0^{5}$ copies of $F_2$, 
where $\delta_{\ref{lem-ES-super}} >0$ depends only on $\varepsilon$. 
Since $n$ is sufficiently large and $q\le \delta \sqrt{n}$, we see that $\tau (G) \ge \delta_{\ref{lem-ES-super}} n_0^5 > {q \choose 2} \lceil \frac{n}{2} \rceil$.
Therefore, we have $|S|\leq \varepsilon^{\frac13} n$.
\end{proof}

\begin{claim}\label{CLA3.4}
Let $R:=R_1\cup R_2$, where
$R_i:=\{v\in V_i: d_{V_i}(v)\geq 2\varepsilon^{\frac13}n\}$.
Then $|R|\leq \frac12\varepsilon^{\frac13}n$.
\end{claim}

\begin{proof}[Proof of claim]
For each $i\in \{1,2\}$, we have 
\begin{equation*}
  e(V_i)=\sum\limits_{v\in V_i}\frac12d_{V_i}(v)\geq
\sum\limits_{v\in R_i}\frac12d_{V_i}(v)\geq|R_i|\varepsilon^{\frac13}n.
\end{equation*}  
Using Claim \ref{CLA3.2} gives
$2\varepsilon n^2\geq e(V_1)+e(V_2)\geq
|R|\varepsilon^{\frac13}n$. 
Thus, we get $|R|\leq 2\varepsilon^{\frac23}n\leq \frac12\varepsilon^{\frac13}n$.
\end{proof}

In what follows, we denote $V^*_i :=V_i\setminus (R\cup S)$ for each $i\in \{1,2\}$. 
Next, we show that the vertices of the partite set $V_i^*$ have a large number of common neighbors in another partite set. 

\begin{claim}\label{CLA3.5} 
If $u_0\in  R_{i}\setminus S$ and
$\{u_1,u_2,u_3,u_4\}\subseteq V^*_{i}$ for some $i\in \{1,2\}$, 
then $u_0,u_1,u_2,u_3,u_{4}$ have at least $\frac{n}{5}$ common neighbors in $V_{3-i}^*$. 
\end{claim}

\begin{proof}[Proof of claim]
Since $u_0\notin S$, we have $d_G(u_0)>(\frac{1}{2}-8\varepsilon^{\frac13})n$.
Since $V_1\cup V_2$
is a partition of $V(G)$ such that $e(V_1,V_2)$
is maximized, we have
$d_{V_{i}}(u_0)\le d_{V_{3-i}}(u_0)$.
Then 
\begin{align}\label{EQU010}
d_{V_{3-i}}(u_0)\geq \frac{1}{2}d_G(u_0)
>\Big(\frac{1}{4}-4\varepsilon^{\frac13}\Big)n.
\end{align}  
For every $u_j\in V^*$, 
we have $d_{V_{i}}(u_j)<2\varepsilon^{\frac13}n$ and $d_G(u_j)>(\frac{1}{2}-8\varepsilon^{\frac13})n$.
Then 
$$ d_{V_{3-i}}(u_j)=d_G(u_j)-d_{V_{i}}(u_j)
> \Big(\frac{1}{2}-10\varepsilon^{\frac13} \Big)n.$$ 
Combining with \eqref{EQU010}, 
we have
\begin{align*}
\Big|\bigcap_{j=0}^{4}N_{V_{3-i}}(u_j)\Big|
&\geq\sum_{j=0}^{4}d_{V_{3-i}}(u_j)-4|V_{3-i}|
>\Big(\frac{1}{4}-4\varepsilon^{\frac13}\Big)n
+4\Big(\frac{1}{2}-10\varepsilon^{\frac13}\Big)n
-4\Big(\frac{1}{2}+\varepsilon^{\frac13}\Big)n\\
&=\Big(\frac{1}{4}-48\varepsilon^{\frac13}\Big)n
\geq \frac{n}{5}.
\end{align*}   
Therefore, there exist at least $\frac{n}{5}$ vertices in ${V}_{3-i}^*$ that are adjacent to $u_0,u_1,u_2,u_3,u_{4}$.
\end{proof}

\begin{claim}\label{CLA3.7}
We have $R\subseteq S$. 
\end{claim}

\begin{proof}[Proof of claim]
Suppose on the contrary that there exists a vertex $u_0\in R\setminus S$.
We assume that $u_0\in R_i\setminus S$ for some $i\in \{1,2\}$.
Then $d_{V_i}(u_0)\geq 2\varepsilon^{\frac{1}{3}}n$
by the definition of $R_i$.
By Claims \ref{CLA3.3} and \ref{CLA3.4},
we have $|R\cup S|\leq\frac32\varepsilon^{\frac13}n.$
Consequently, we get $d_{V_i}(u_0)\geq|R\cup S|+ \frac{1}{2}\varepsilon^{\frac{1}{3}} n$.  
So there are at least $\frac{1}{2}\varepsilon^{\frac{1}{3}} n$ neighbors of $u_0$ outside $R\cup S$. 
For any two such neighbors, say $v_1,v_2 \in V_i^*$, we count the number of copies of $F_2$ spanned by them. 
By Claim \ref{CLA3.5}, the vertices
$u_0$, $v_1$ and $v_2$ have at least $\frac{n}{5}$ common neighbors in $V^*_{3-i}$.
Therefore, there are at least $\frac{n}{5}(\frac{n}{5}-1)$ copies of $F_2$
containing $u_0v_1$ and $u_0v_2$. 
In total, we can find $\tau (G)\ge {\frac{1}{2}\varepsilon^{{1}/{3}}n \choose 2} \cdot \frac{n}{5}(\frac{n}{5} -1) \ge \frac{1}{500}\varepsilon^{\frac{2}{3}}n^4$, 
which contradicts with the assumption $\tau (G)\le {q \choose 2} \lfloor \frac{n}{2} \rfloor$.
Thus, we get $R\subseteq S$. 
\end{proof}

\begin{claim} \label{CLA3.7-b}
    For every $i\in \{1,2\}$,
the edges in $G[V_i^*]$ are pairwise disjoint and $e(V_i^*)< 2q$.
\end{claim}

\begin{proof}[Proof of claim] 
First, we prove that $G[V_i^*]$ consists of pairwise disjoint edges and isolated vertices. 
Indeed, if $v_1u_0$ and $u_0v_2$ are two incident edges in $G[V_i]$, then Claim \ref{CLA3.5} implies that $v_1,v_2,u_0$ have at least $\frac{n}{5}$ common neighbors in $V_{3-i}^*$, which leads to $\frac{n}{5}(\frac{n}{5}-1)$ copies of $F_2$, a contradiction.  

Now we show that $e(V_i^*)< 2q$ for each $i\in \{1,2\}$. 
Suppose on the contrary that there exist at least $2q$ edges $u_1v_1,\dots,u_{2q} v_{2q}$ in $G[V_i^*]$.
For any two distinct edges $u_{j_1}v_{j_1}$ and $u_{j_2}v_{j_2}$ in $G[V_i^*]$,
by Claim \ref{CLA3.5}, 
$u_{j_1},v_{j_1},u_{j_2}$ and $v_{j_2}$ have at least $\frac{n}{5}$ common neighbors in $V^*_{3-i}$. 
In other words, any two distinct edges of $G[V_i^*]$ lead to at least $\frac{n}{5}$ copies of $F_2$ with centers in $V_{3-i}^*$.
Thus, we get $\tau(G)\geq \binom{2q}{2}\frac{n}{5}>\binom{q}{2}\lceil\frac{n}{2}\rceil$,
which is a contradiction.
\end{proof}

Recall that $u^*$ is a vertex of $G$ with  $x_{u^*}=\max_{v\in V(G)}x_v$. 
By symmetry, 
we may assume that $u^{*} \in V_1$. 
Since $\lambda (G)x_{u^*}= \sum_{v\in N_G(u^*)}x_v \le d_G(u^*)x_{u^*}$, we get $d_G(u^*)\ge \lambda (G)> \frac{n}{2}$. Therefore, we have $u^*\notin S$, which together with Claim \ref{CLA3.7} yields $u^*\notin R$ and $d_{V_1}(u^*) <2\varepsilon^{\frac{1}{3}} n$.

\begin{claim} \label{EQU013}
    There is a partition $V_2^*=A\sqcup B$ such that $|A|=\lfloor \varepsilon^{\frac13} n\rfloor$ and $B$ is an independent set of $G$. Moreover, we have $\sum_{v\in B} x_v \ge \big(\lambda (G)- 4\varepsilon^{\frac{1}{3}}n \big)x_{u^*}$. 
\end{claim}

\begin{proof}[Proof of claim] 
By Claim \ref{CLA3.7-b}, we know that $G[V_{2}^{*}]$ 
contains less than $2q$ edges.  
Therefore, we can find a subset $A\subseteq V_2^*$ 
with $|A|=\lfloor \varepsilon^{\frac13} n\rfloor$ such that all edges of $V_2^*$ lie in $A$. We denote $B:=V_2^*\setminus A$. Then $B$ is an independent set of $G[V_2]$.   
Combining with Claim \ref{CLA3.3}, we have
\begin{align*}
 \lambda (G)x_{u^{*}}&\leq \sum_{v\in N_{S}(u^{*})}x_v+ \sum_{v\in N_{V^*_{1}}(u^{*})}x_v 
 + \sum_{v\in  A}x_v+
\sum_{v\in B}x_v
                     \nonumber\\
&<  |S|x_{u^{*}}+2\varepsilon^{\frac13} nx_{u^{*}}
 + \varepsilon^{\frac13} nx_{u^{*}}+\sum_{v\in B}x_v\\
&<  4\varepsilon^{\frac13} nx_{u^{*}}
+\sum_{v\in  B}x_v. 
\end{align*}  
Consequently, we have $\sum_{v\in  B}x_v 
\geq \big(\lambda (G)-4\varepsilon^{\frac13} n\big)x_{u^{*}}$, as needed. 
\end{proof}

\begin{claim}\label{CLA3.8}
We have $S=\varnothing$.
\end{claim}

\begin{proof}[Proof of claim]
Suppose on the contrary that $S\neq \varnothing$.
Let $G'$ be the graph obtained from $G$ by deleting all edges incident to the vertices of $S$, and then adding all edges between $S$ and $B$. 
Next, we show that $\lambda (G')>\lambda (G)$.
For any $u\in S$, we have 
$d_G(u)\leq (\frac{1}{2}-8\varepsilon^{\frac13})n$.
Then 
\[  \sum_{v\in N_G(u)}x_{v} 
\le  \Big(\frac{n}{2} - 8\varepsilon^{\frac{1}{3}}n \Big) x_{u^*} \leq\big(\lambda (G)-8\varepsilon^{\frac13} n\big)x_{u^{*}}. \]
Combining with Claim \ref{EQU013} gives
\begin{align*}
  \lambda (G')-\lambda (G) \geq \bm{x}^{\mathrm{T}}\big(A(G')-A(G)\big)\bm{x}
                  \geq 2\sum\limits_{u\in S}x_{u}\Big(\sum\limits_{v\in B}x_v-\sum\limits_{v\in N_G(u)}x_v\Big)>0.
\end{align*}  
Hence, we get  $\lambda (G) < \lambda (G')$.

In what follows, we are ready to show $\lambda (G')< \lambda (T_{n,2})$. 
Set $V'_{1}:=V_{1}\cup S$
and $V'_{2}:=V_{2}\setminus S$.
Let $K'$ be the complete bipartite graph with partite sets $V_1'$ and $V_2'$.
We see that $G'$ can be obtained from $K'$ by adding at most $ e(V_1^*)+e(V_2^*)\leq 4q$ class-edges, and deleting at least 
 $|S||A| \ge \lfloor\varepsilon^{\frac13} n\rfloor$ cross-edges, since there is no edge between $S$ and $A$. 
Note that $\lfloor\varepsilon^{\frac13} n\rfloor \gg 4q$ for sufficiently large $n$. Let $G''$ 
be the graph obtained from $G'$ by adding some missing edges such that $G''$ misses exactly $8q$ cross-edges. Since $G''$ is a subgraph of $G'$, we get $\lambda (G') \le \lambda (G'')$. 
Finally, we apply Theorem \ref{first-key} to show that $\lambda (G'') < \lambda (T_{n,2})$. Indeed,  setting $\alpha_1=e(V_1^*)+ e(V_2^*) \le 4q$,  $\alpha_2=8q$ and $k=0$ in Theorem \ref{first-key} (ii), we have $\psi =2(\alpha_1+ \alpha_2)\le 24q$  and then  
\[ \lambda (G'') \le \lambda (T_{n,2}) + \frac{-8q}{n}+ \frac{56\cdot 12q\cdot 14\cdot 24q}{n^2} < \lambda (T_{n,2}) - \frac{q}{n}, \]
where the last inequality holds since $q \le 10^{-4} \sqrt{n}$. So we get $\lambda (G'')< \lambda (T_{n,2})$. Consequently, it follows that $\lambda (G) \le \lambda (G'')< \lambda (T_{n,2})$, a contradiction. So we conclude that $S=\varnothing$. 
\end{proof}

By Claims \ref{CLA3.7} and \ref{CLA3.8}, we have $R=S=\varnothing$,
and thus $V_i=V_i^*$ for each $i\in \{1,2\}$.
We denote $n_1=|V_1|$ and $n_2=|V_2|$. 
Let $K_{n_1,n_2}$ be the complete bipartite graph on the partite sets $V_1$ and $V_2$. 
Let $G_{in}$ and $G_{cr}$ be
the graphs induced by edges
in $E(G)\setminus E(K)$
and $E(K)\setminus E(G)$, respectively.
We denote $\alpha_1=e(G_{in})$ and $\alpha_{2}=e(G_{cr})$.

\begin{claim}\label{CLA3.8T}
The graph $G$ is obtained from the complete bipartite graph $K_{n_1,n_2}$ by adding $\alpha_1$ class-edges, and deleting $\alpha_2$ cross-edges, where  
$\alpha_2\leq \alpha_1\leq 4q$ and $|n_1-n_2|< 4\sqrt{\alpha_1}$.
\end{claim}

\begin{proof}[Proof of claim]
By Claims \ref{CLA3.7} and \ref{CLA3.8}, we have $\alpha_1= e(V_1)+e(V_2)\leq 4q$.
Next, we show that $\alpha_2\leq \alpha_1$.
Otherwise, if $\alpha_2\geq \alpha_1+1$, then let $G'$ be the graph obtained from $G$ by adding some missing cross-edges so that $G'$ misses exactly $\alpha_1+1$ cross-edges. As $G$ is a subgraph of $G'$, it follows that $\lambda (G) \le \lambda (G')$. Since $q\le {10^{-4}} \sqrt{n}$, setting $k=0$ in Theorem \ref{first-key} (ii), we obtain $\lambda (G')\le  \lambda (T_{n,2}) +\frac{-2}{n}+\frac{0.1}{n} < \lambda (T_{n,2})$, which yields $\lambda (G) < \lambda (T_{n,2})$, a contradiction. 

Now, we show that $|n_1-n_2|< 4\sqrt{\alpha_1}$.
Otherwise, if $|n_1-n_2|\geq 4\sqrt{\alpha_1}$, 
then we set $k:=\lfloor 2\sqrt{\alpha_1}\rfloor$ in Theorem \ref{first-key} (ii), and we get $\psi :=\max\{3k, 2(\alpha_1+\alpha_2)\} \le 16q$ and   
$\lambda (G)\leq \lambda (T_{n,2})+\frac{2\alpha_1}{n}- \frac{k^2}{n}(1-o(1)) +\frac{0.1}{n} < \lambda (T_{n,2}) $, which is a contradiction. Thus, we conclude that $n_1-n_2 < 4\sqrt{\alpha_1}$. 
\end{proof}

Now, we are in the final stage of the proof of Theorem \ref{second-key}. Invoking Claim \ref{CLA3.8T}, we see that both Theorem \ref{first-key} (i) and Lemma \ref{LEM2.9A} are applicable. 
In what follows, we refine the bounds obtained from Claim \ref{CLA3.8T}. 
Assume that $n_1-n_2:=s$ for some integer $s\ge 0$. 
Then $n_1=\frac{n+s}{2}$ and $n_2=\frac{n-s}{2}$. 
By Claim \ref{CLA3.8T}, we see that $s\le 8\sqrt{q}=O(n^{1/4})$.  
By computation, it follows that 
$$\lambda(K_{n_1,n_2})= \sqrt{n_1n_2}= \frac{n}{2}-\frac{s^2}{4n}+O\Big(\frac{1}{n^{2}}\Big),$$
which together with Theorem \ref{first-key} (i) yields  
$$ \lambda (G) \le \frac{n}{2}-\frac{s^2}{4n}+\frac{2(\alpha_1-\alpha_2)}{n} + \frac{0.1}{n}
+O\Big(\frac{1}{n^{2}}\Big).$$
Combining with \eqref{EQU009DD}, we see that if $n$ is even, then $\frac{s^2}{4n}\le \frac{2(\alpha_1-\alpha_2)}{n} - \frac{2}{n} + \frac{0.1}{n}$, which gives $s^2\le 8(\alpha_1-\alpha_2-1)$; if $n$ is odd, then $\frac{s^2}{4n} \le 
\frac{2(\alpha_1-\alpha_2)}{n} - \frac{7}{4n} + \frac{0.1}{n}$, so $s^2\le 8 (\alpha_1-\alpha_2) -7$.  
Consequently, we have $\alpha_2 <\alpha_1$.
Finally,
it suffices to show that $\alpha_1\leq q$.
Otherwise, if $\alpha_1\ge q+1$, then by Lemma \ref{LEM2.9A}, we get 
$\tau(G)\geq \binom{q+1}{2}\frac{n-s}{2}- \alpha_1\alpha_2
\ge {q \choose 2}\frac{n+1}{2} + q \frac{n}{2} -23q^{5/2} >{q \choose 2} \big\lceil \frac{n}{2} \big\rceil$, where the last inequality holds since $q \le 10^{-2} n^{2/3}$.  
This leads to a contradiction. Thus, we get $\alpha_1 \le q$. 
\end{proof}

\section{Proofs of Theorems \ref{THM1.6A} and \ref{THM1.4A}}

\label{section4}

In this section, 
we prove Theorems \ref{THM1.6A} and \ref{THM1.4A}. 
Recall that $K_{\lceil\frac{n}{2}\rceil, \lfloor\frac{n}{2}\rfloor}^{q}$ denotes the graph obtained from $K_{\lceil\frac{n}{2}\rceil, \lfloor\frac{n}{2}\rfloor}$ by adding $q$ pairwise disjoint edges to the partite set of size $\lceil \frac{n}{2} \rceil$.

\begin{proof}[\emph{\textbf{Proof of Theorem \ref{THM1.6A}}}]
Let $\delta =\delta_{\ref{second-key}}$ 
be a fixed real number, where $\delta_{\ref{second-key}} > 0$ is determined in Theorem \ref{second-key}. 
Assume that $G$ is an $n$-vertex graph with $\lambda (G)\geq \lambda (K_{\lceil\frac{n}{2}\rceil,\lfloor\frac{n}{2}\rfloor}^{q})$, where $2\le q\le \delta \sqrt{n}$.  
We may assume that $G$ minimizes the number of copies of the bowtie. 
Then $\tau (G)\leq \tau(K_{\lceil\frac{n}{2}\rceil,\lfloor\frac{n}{2}\rfloor}^{q})
= {q\choose 2} \lfloor \frac{n}{2}\rfloor$. 
Since $q\ge 2$, by Theorem \ref{first-key} (i), it is easy to see that $\lambda (G)\ge \lambda (K_{\lfloor\frac{n}{2}\rfloor,\lceil\frac{n}{2}\rceil}^{1})$. 
Applying Theorem \ref{second-key}, we know that $G$ is obtained from $K_{n_1,n_2}$ by adding $\alpha_1$ class-edges and deleting $\alpha_2$ cross-edges, where $\alpha_2< \alpha_1 \le q$ and $n_1-n_2\le \sqrt{8\alpha_1}$. 
By Theorem \ref{first-key} (ii), we have
$$ \lambda (G)\leq \lambda (T_{n,2})+\frac{2(\alpha_1-\alpha_2)}{n}+\frac{0.1}{n}.$$ 
Using the previous assumption and Theorem \ref{first-key} (i), it follows that 
$$ \lambda (G)\ge \lambda (K_{\lceil\frac{n}{2}\rceil,\lfloor\frac{n}{2}\rfloor}^{q})
\ge \lambda (T_{n,2})+\frac{2q}{n}- 
\frac{0.1}{n}.$$
Combining with the above bounds, we get $\alpha_2+q\le \alpha_1 + 0.1$.
So $\alpha_2=0$ and $\alpha_1=q$. 
We conclude that $G$ is obtained from $K_{n_1,n_2}$ by adding $q$ class-edges. 
Using Theorem \ref{first-key} (ii) again, we see that $|n_1-n_2|\le 1$. 
A similar argument of Lemma \ref{LEM2.9A} shows  that 
all the $q$ class-edges are pairwise disjoint,  and they are added into the same partite set of $K_{n_1,n_2}$.  Consequently,  if $n$ is even, then $G=K_{\frac{n}{2},\frac{n}{2}}^{q}$; if $n$ is odd, then 
$G= K_{\frac{n+1}{2},\frac{n-1}{2}}^{q}$ since
$\tau(G)\le \binom{q}{2}\lfloor\frac{n}{2}\rfloor$. 
\end{proof}

Next, we provide the proof of Theorem \ref{THM1.4A}. 

\begin{proof}[\emph{\textbf{Proof of Theorem \ref{THM1.4A}}}]
Suppose that $\lambda (G)\ge \lambda (K_{\lfloor\frac{n}{2}\rfloor,\lceil\frac{n}{2}\rceil}^{1})$
and $G \neq K_{\lfloor\frac{n}{2}\rfloor,\lceil\frac{n}{2}\rceil}^{1}$. 
Our goal is to show that $\tau (G)\ge \lfloor \frac{n-1}{2}\rfloor$. 
We may assume that $G$ minimizes the number of copies of the bowtie. 
The proof will be divided into two cases based on the parity of $n$.

\medskip 
{\textbf{Case 1.}} $n$ is even. 

By Lemma \ref{lem-even-case}, we have 
$\tau(G)\leq \tau(K_{\frac{n}{2}+1,\frac{n}{2}-1}^{2})=\frac{n}{2} -1$. 
Using Theorem \ref{second-key}, 
we know that $G$ is obtained from $K_{n_1,n_2}$ by adding $\alpha_1$ class-edges and deleting $\alpha_2$ cross-edges, where $\alpha_2< \alpha_1 \le 2$ and $(n_1-n_2)^2\le 8(1-\alpha_2)$.  
Thus, we obtain that $(\alpha_2,n_1-n_2)\in \{(0,0),(0,2),(1,0)\}$.

If $\alpha_1\le 1$, then $G$ is $F_2$-free. 
Since $\lambda (G)\geq \lambda (K_{\lfloor\frac{n}{2}\rfloor,\lceil\frac{n}{2}\rceil}^{1})$, Theorem \ref{thm-LLP} implies $G=K_{\lfloor\frac{n}{2}\rfloor,\lceil\frac{n}{2}\rceil}^{1}$, 
which contradicts with the assumption. 
Thus, we have $\alpha_1=2$. 
By Lemma \ref{LEM2.9A},
 the two class-edges lie in the same partite set of $K_{n_1,n_2}$. 
Therefore, it follows that 
$$G\in \{
K_{\frac{n}{2},\frac{n}{2}}^{2},
K_{\frac{n}{2}-1,\frac{n}{2}+1}^{2},
K_{\frac{n}{2}+1,\frac{n}{2}-1}^{2},
K_{\frac{n}{2},\frac{n}{2}}^{2, \Gamma},
K_{\frac{n}{2},\frac{n}{2}}^{2,1}\}.$$
Since $\tau(G)\leq \frac{n}{2}-1$, 
we see that  
$$ G\neq K_{\frac{n}{2},\frac{n}{2}}^{2},
K_{\frac{n}{2}-1,\frac{n}{2}+1}^{2},
K_{\frac{n}{2},\frac{n}{2}}^{2,1}. $$ 
From Lemma \ref{lem-even-case}, 
we get $G=K_{\frac{n}{2}+1,\frac{n}{2}-1}^{2}$ or 
$K_{\frac{n}{2},\frac{n}{2}}^{2, \Gamma}$, as desired.

\medskip 
{\textbf{Case 2.}} $n$ is odd.

From Lemma \ref{lem-odd-case}, we see that $\tau (G)\le \tau (K_{\frac{n+1}{2},\frac{n-1}{2}}^2) = \frac{n-1}{2}$. Similarly, Theorem \ref{second-key} implies that $G$ is constructed from $K_{n_1,n_2}$ by adding $\alpha_1$ class-edges and deleting $\alpha_2$ cross-edges, where $\alpha_2<\alpha_1=2$ and $(n_1-n_2)^2 \le 9- 8\alpha_2$. 
Since $n_1-n_2$ is odd, it follows that $(\alpha_2,n_1-n_2)\in \{(0,1),(0,3),(1,1)\}$.
Therefore, we have the following two cases: 
\begin{itemize}
\item 
For $\alpha_2=0$, we have $G\in \{K_{\frac{n+1}{2},\frac{n-1}{2}}^{2},
K_{\frac{n-1}{2},\frac{n+1}{2}}^{2},
K_{\frac{n+3}{2},\frac{n-3}{2}}^{2},
K_{\frac{n-3}{2},\frac{n+3}{2}}^{2}\}$; 

\item 
For $\alpha_2=1$, we have 
$G\in \{
K_{\frac{n+1}{2},\frac{n-1}{2}}^{2, \Gamma}, 
K_{\frac{n+1}{2},\frac{n-1}{2}}^{2,1},
K_{\frac{n-1}{2},\frac{n+1}{2}}^{2, \Gamma}, 
K_{\frac{n-1}{2},\frac{n+1}{2}}^{2,1}\}$. 
\end{itemize}

Since $\lambda (G)\geq\lambda (K_{\lfloor\frac{n}{2}\rfloor,\lceil\frac{n}{2}\rceil}^{1})$, 
Lemma \ref{lem-odd-case} (ii) implies 
that  
$$ G\neq K_{\frac{n+3}{2},\frac{n-3}{2}}^{2},K_{\frac{n+1}{2},\frac{n-1}{2}}^{2, \Gamma}. $$ 
On the other hand, due to $\tau(G)\leq \frac{n-1}{2}$, it follows that 
$$ G\neq K_{\frac{n-1}{2},\frac{n+1}{2}}^{2}, K_{\frac{n-3}{2},\frac{n+3}{2}}^{2}, K_{\frac{n-1}{2},\frac{n+1}{2}}^{2,1}. $$ 
From Lemma \ref{lem-odd-case} (i), we conclude that 
$G=K_{\frac{n+1}{2},\frac{n-1}{2}}^{2},K_{\frac{n+1}{2},\frac{n-1}{2}}^{2,1}$ or $
K_{\frac{n-1}{2},\frac{n+1}{2}}^{2, \Gamma}$, as needed. 
\end{proof}

\section{Concluding remarks}

\label{sec-Concluding}

\iffalse 
For an integer $k\geq 1$, define
$$
f(k)=
\begin{cases}
k^2 - k, & \text{if } k \text{ is odd}; \\
k^2 - \frac{3}{2}k, & \text{if } k \text{ is even}.
\end{cases}
$$
In 1995, Erd\H{o}s, F\"{u}redi, Gould and Gunderson \cite{Erdos1995} proved that: For every \( k \geq 1 \) and \( n \geq 50k^2 \),
we have $\mathrm{ex}(n, F_k) =\lfloor \frac{n^2}{4}\rfloor +f(k)$.
Let $\tau_k (G)$ be the number of copies of $F_k$ in $G$.
Inspired by the above result, we propose the following conjecture.

\begin{conj}\label{conj5.1}
Let $k\geq 1$.
There exists an absolute constant $\delta>0$ such that 
for any sufficiently large $n$ and $f(k)+1\le q \le \delta \sqrt{n}$, the following holds: If $G$ is an $n$-vertex graph with $\lambda (G)\geq \lambda (K_{\lceil\frac{n}{2}\rceil,\lfloor\frac{n}{2}\rfloor}^{q})$,
then $\tau_k(G)\geq \binom{q}{k}\lfloor\frac{n}{2}\rfloor$,
where the equality holds if and only if $G= K_{\lceil\frac{n}{2}\rceil,\lfloor\frac{n}{2}\rfloor}^{q}$.
\end{conj}

\fi 

In this paper, we studied the supersaturation problem for the bowtie in terms of the spectral radius.  
We solved a conjecture of Li--Feng--Peng \cite{LFP-bowtie} by showing that if $n$ is sufficiently large and $G$ is an $n$-vertex graph with $\lambda (G)\ge \lambda (K_{\lceil\frac{n}{2}\rceil,\lfloor\frac{n}{2}\rfloor}^{q})$, where $2\le q\le \delta \sqrt{n}$ for some absolute constant $\delta >0$, then $\tau (G)\ge {q \choose 2} \lfloor \frac{n}{2}\rfloor$, with equality if and only if $G=K_{\lceil\frac{n}{2}\rceil,\lfloor\frac{n}{2}\rfloor}^{q}$. 
In fact, 
by the definition of $K_{\lceil\frac{n}{2}\rceil,\lfloor\frac{n}{2}\rfloor}^{q}$,
it can be seen that the best possible range for 
$q$ satisfies $q\leq \lfloor\frac{n+1}{4}\rfloor$. 
It would therefore be interesting to investigate whether Theorem \ref{THM1.6A} remains valid over this wider range. 

\begin{conj}
For sufficiently large $n$ and $2\le q \le \lfloor\frac{n+1}{4}\rfloor$, if $G$ is an $n$-vertex graph with $\lambda (G)\geq \lambda (K_{\lceil\frac{n}{2}\rceil,\lfloor\frac{n}{2}\rfloor}^{q})$,
then $\tau(G)\geq \binom{q}{2}\lfloor\frac{n}{2}\rfloor$,
with equality if and only if $G= K_{\lceil\frac{n}{2}\rceil,\lfloor\frac{n}{2}\rfloor}^{q}$.
\end{conj}

At the end of this paper, we conclude another type of spectral extremal problem for graphs with given size $m$, instead of the order $n$.  
A well-known result of Nosal \cite{Nosal1970} and Nikiforov \cite{Niki2002cpc} states that if $G$ is a triangle-free graph with $m$ edges, then $\lambda (G)\le \sqrt{m} $, 
with equality if and only if
$G$ is a complete bipartite graph. 
In 2023, Ning and Zhai \cite{NZ2023} proved that if $\lambda (G)\ge \sqrt{m}$, then $G$ contains at least $\lfloor \frac{\sqrt{m}-1}{2} \rfloor$ triangles, unless $G$ is a complete bipartite graph. For related edge-spectral extremal results, we refer to \cite{LLZ2024-book-4-cycle,LLZ-edge-spectral,LLZ-edge-color-critical} and references therein.

In 2023, 
Li, Lu and Peng \cite{Li2023} proved that for every integer $m\ge 8$, if $G$ is an $F_2$-free graph with $m$ edges, then $\lambda (G)\le \frac{1+\sqrt{4m-3}}{2}$, 
with equality if and only if $G = K_2\vee \frac{m-1}{2}K_1$, which is the join obtained from an edge $K_2$ and an independent set ${\frac{m-1}{2}}K_1$. Consequently, we are led to the following supersaturation problem under the edge-spectral condition.

\begin{prob}
Given an $m$-edge graph $G$ with $\lambda(G) > \frac{1+\sqrt{4m-3}}{2}$, what is the best possible lower bound on the number of copies of the bowtie $F_2$ in $G$?
\end{prob}

\end{document}